\documentclass[a4paper,reqno]{amsart}

\usepackage{fullpage}
\usepackage{verbatim}
\usepackage{amsfonts}
\usepackage{mathrsfs}
\usepackage{stmaryrd}
\usepackage{graphicx}
\usepackage{mathdots}
\usepackage{float}
\usepackage[dvipsnames]{xcolor}
\usepackage[colorlinks = true, linkbordercolor = {white}, urlcolor={purple}, linkcolor={purple}, citecolor = {purple}]{hyperref}
\usepackage{cleveref}
\usepackage{soul}
\usepackage{enumitem}
\usepackage{amssymb}
\usepackage{hyperref}
\usepackage[foot]{amsaddr}
\usepackage{caption}
\usepackage{microtype}

\theoremstyle{plain}
    \newtheorem{Thm}{Theorem}[section]
    
    \newtheorem{Lem}[Thm]{Lemma}

\theoremstyle{definition}
    \newtheorem{Def}[Thm]{Definition}

\theoremstyle{remark}
    
\theoremstyle{remark}

\Crefname{Thm}{Theorem}{Theorems}
\Crefname{Prop}{Proposition}{Propositions}
\Crefname{Lem}{Lemma}{Lemmas}
\Crefname{Cor}{Corollary}{Corollaries}
\Crefname{Def}{Definition}{Definitions}
\Crefname{Rmk}{Remark}{Remarks}
\Crefname{notation}{Notation}{Notations}
\Crefname{figure}{Figure}{Figures}

\begin{document}

\title[Cut and project schemes in the Poincar\'e disc]{Cut and project schemes in the Poincar\'e disc:\\[0.25em]
From cocompact Fuchsian groups to chaotic Delone sets}

\author{Richard A. Howat}
\address{Richard A.\,Howat --
ORCID: 0009-0009-8008-8922\newline
\parbox{2em}{\mbox{ }}email: rah955@student.bham.ac.uk\newline 
\parbox{2em}{\mbox{ }}address: School of Mathematics, University of Birmingham, UK}

\author{Tony Samuel}
\address{Tony Samuel: ORCID: 0000-0002-5796-0438\newline
\parbox{2em}{\mbox{ }}email: a.samuel@exeter.ac.uk\newline
\parbox{2em}{\mbox{ }}address: Department of Mathematics and Statistics, University of Exeter, UK}

\author{Ay\c{s}e Y{\i}ltekin-Karata\c{s}}
\address{Ay\c{s}e Y{\i}ltekin-Karata\c{s} -- ORCID: 0009-0003-6838-0759 \newline
\parbox{2em}{\mbox{ }}email: akaratas@bartin.edu.tr\newline
\parbox{2em}{\mbox{ }}address: Department of Mathematics, Bartın University, Türkiye}


\begin{abstract}
A question raised by Davies \textsl{et al} [\textsl{Phys.\,Rev.\,Lett.}\,{\bfseries 131}, 2023] is:\ \textsl{Can developing new cut and project models, where the lattice is not square or the curve is non-linear, generate better performing graded metamaterials?} In this article, we study a natural construction of such a cut and project scheme, namely, cut and project schemes in relation to cocompact Fuchsian groups acting on the Poincaré disc model of hyperbolic space. We present a condition on the fundamental domain (a hyperbolic polygon) of the group so that the resulting cut and project set $S \subset \mathbb{R}$ is a chaotic Delone set. We also investigate the set of tile lengths of $S$, namely $\mathcal{L}_{S} = \{ z - y : z,y \in S, \, z > y \; \text{and} \; (y, z) \cap S = \emptyset \}$, and show that this set is countably infinite.
Finally, we apply our results to cocompact Fuchsian triangle groups and show that the resulting cut and project sets are chaotic Delone, complementing and extending the work of L\'{o}pez \textsl{et al.} [\textsl{Discrete Contin. Dyn. Syst.}  {\bfseries 41}, 2021].
\end{abstract}

\maketitle

\section{Introduction}
Central to understanding properties of physical quasicrystals and metamaterials (deliberately engineered composites with bespoke wave-control properties that arise from their structure) is the careful analysis of point sets arising from, for instance, cut and project schemes in Euclidean space \cite{Aperiodic_Order_Vol1, GL_1989}. Here, loosely speaking, one takes a straight line cutting through a square lattice, projects a given window of lattice points onto a straight line $l$ and considers the resulting point set $\Lambda \subset l$, see \Cref{Fig1}. Under this, and similar constructions, the given point set is a \emph{Delone set}, namely uniformly discrete and relatively dense.  In \cite{PhysRevLett.131.177001}, this method is adapted by examining what happens when instead of the straight line $l$ one takes a quadratic curve. Using this new construction, the authors of \cite{PhysRevLett.131.177001}, built new acoustic metamaterials with large spectral gaps, yielding high-performing sound insulators.  Moreover, a question raised in \cite{PhysRevLett.131.177001} is: \textsl{Can developing new cut and project models, where the lattice is not square or the curve is non-linear, generate better performing graded metamaterials?}

    \begin{figure}[hb!]
    \begin{center}
    \includegraphics[scale=1]{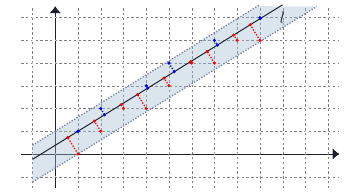} 
    \caption{Example of a cut and project scheme in $\mathbb{R}^{2}$.}
\label{Fig1}
    \end{center}
    \end{figure}

A natural construction of such a cut and project scheme can be found in \cite{ChaoticDeloneSet}, where the authors investigate point sets of cut and project schemes in the Poincar\'e disc model of hyperbolic space, and which were found, under certain hard to verify conditions, to be an interesting form of Delone set called \textit{chaotic Delone} (see \Cref{def:ChaoticDeloneSet}).

To define cut and projection schemes in Poincar\'e disc model of hyperbolic space, we require the following notation. Throughout, we denote the Poincaré disc by $\mathbb{D}=\{z\in\mathbb{C}:|z|<1\}$ and endow it with the hyperbolic metric $d$ induced by the differential
\begin{align}\label{eq:hyperbolic_differntial}
    ds = \frac{2|dz|}{1-|
    z|^2}.
\end{align}
We assume that geodesics in $\mathbb{D}$ with respect to $d$ are parametrised by arc length (namely, traversed with unit speed) and we orient the normal bundle of a geodesic $k$ by the standard orientation of the Poincaré disc and the orientation of the geodesic $k$.

Let $\Gamma$ be a Fuchsian group with fundamental domain $\tau$, and let $x\in\tau$. Let $k$ denote a geodesic in $\mathbb{D}$ with parametrisation $\kappa : \mathbb{R} \to \mathbb{D}$. For $\rho \in \mathbb{R}$ positive, set $\mathcal{N}(k,\rho) = \{x\in \mathbb{D}:d(x,k)<\rho\}$ and let $\partial^+ \mathcal{N}(k,\rho)$ be the connected component of the boundary of $\mathcal{N}(k,\rho)$ which is positive with respect to the normal bundle of $k$. We define 
    \begin{align}\label{eq:N+}
    \overline{\mathcal{N}(k,\rho)}^+=\mathcal{N}(k,\rho)\cup \partial^+\mathcal{N}(k,\rho)
    \quad \text{and} \quad
    S_{\mathbb{D}}^+(k,\rho,x)=p_{k} (\Gamma (x) \cap\overline{\mathcal{N}(k,\rho)}^+),
    \end{align}    
where $\Gamma(x)$ denotes the orbit of $x$ under the action of $\Gamma$ on $\mathbb{D}$ and $p_k:\mathbb{D}\rightarrow k$ denotes the orthogonal projection of $\mathbb{D}$ onto $k$. We call $S_{\mathbb{D}}^+(k,\rho,x)$ a \textit{hyperbolic cut and project set}, and let 
    \begin{align*}
    S^+(k,\rho,x) = \kappa^{-1}(S_{\mathbb{D}}^+(k,\rho,x)) \subseteq \mathbb{R}.
    \end{align*}
In \cite{ChaoticDeloneSet}, the sets $S_{\mathbb{D}}^+(k,\rho,x)$ are studied when $\Gamma$ is a torsion free cocompact discrete subgroup of $\operatorname{Isom}^{+}(\mathbb{D})$, which by \cite[Definition 9.17 and Lemma 11.9]{EinsiedlerWard2011}
is equivalent to $\Gamma$ being uniform lattices. Indeed, letting 
    \begin{align*}
       \operatorname{inj}(\Gamma,x)=\frac{1}{2}\inf \{d(y,z)\mid y,z\in \Gamma(x), \, y\neq z\} 
    \end{align*}
    denote the \textit{injectivity radius} of $\Sigma_\Gamma=\Gamma\backslash\mathbb{D}$ with respect to $x$, it is found, taking a geodesic $\ell$ which contains a point whose orbit under the geodesic flow is dense in the unit tangent bundle of $\Gamma \backslash \mathbb{D}$, $x\in\tau$ and $\rho\in (0,\operatorname{inj}(\Gamma,x))$, that $S^+(k,\rho,x)$ is chaotic Delone if and only if the following condition is satisfied.

\begin{figure}[ht]
 \includegraphics[scale=0.33]{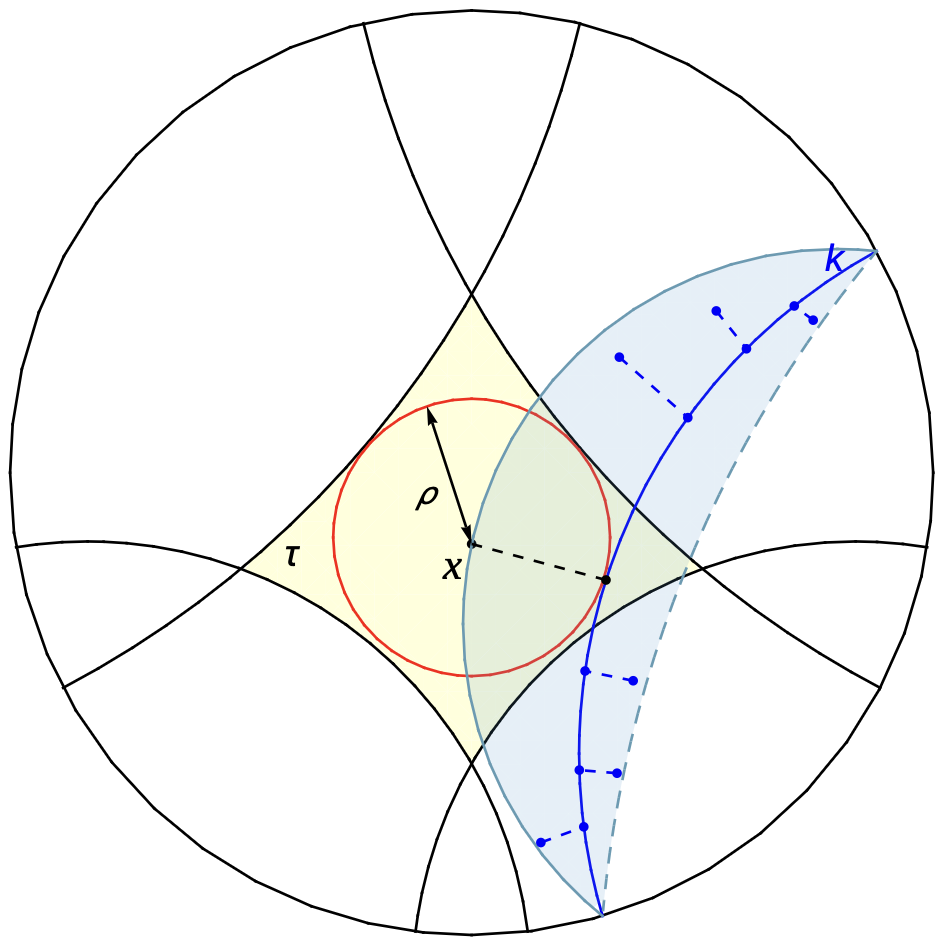}
 \caption{A hyperbolic cut and project scheme for the triangle group with signature (6,6,3) and quadrilateral fundamental domain.}
 \label{f:Cut and Project}
\end{figure}

{\em Condition\,B: Any geodesic on $\Sigma_\Gamma$ intersects the closed disc $\Delta$ of radius $\rho$ centred at $\pi(x)$, and there exists no geodesic with one-sided tangency with $\partial\Delta$.}

Here, $\pi$ denotes the natural quotient map from $\mathbb{D}$ to $\Gamma\backslash\mathbb{D}$. 
It is stated in \cite{ChaoticDeloneSet} that this condition is highly non-trivial to verify. Our main result, \Cref{Thrm-Delone if and only if one sided tangency} resolves this by relaxing the torsion free assumption and gives an easy to verify condition on the fundamental domain of a cocompact Fuchsian group for when $S^+(k,\rho,x)$ is chaotic Delone set, where $x$ is the hyperbolic centre of the fundamental domain, namely the point that is equidistant, with respect to $d$, from each side of $\tau$. In particular, in the setting of cocompact triangle group we obtain the following.

\begin{Thm}\label{T:MainResult}
    Let $\Gamma$ be a cocompact triangle group with signature $(m_1, m_2, m_3)$, $\mathcal{F}$ be a fundamental domain for $\Gamma$ with centre point $x$. Moreover, let $\ell$, be a geodesic whose orbit under the geodesic flow is dense in the unit tangent bundle of $\mathcal{F}$. 
    \begin{enumerate}
        \item If $\mathcal{F}$ is a quadrilateral, then there exists $\rho>0$ such that $S^+(\ell,\rho,x)$ is chaotic Delone if and only if there are at least two odd numbers in the signature of $\Gamma$.
        \item If $\mathcal{F}$ is a hexagon, then there exists $\rho>0$ such that $S^+(\ell,\rho,x)$ is chaotic Delone.
        \item The set of tile lengths of $S^+(\ell,\rho,x)$, namely the set
            \begin{align*}
            \mathcal{L}_{S^+(\ell,\rho,x)} = \{ z-y : z,y\in S^+(\ell,\rho,x), \, z > y \; \text{and} \; (y,z) \cap S^+(\ell,\rho,x) = \emptyset
    \},
    \end{align*}
        is countably infinite.
    \end{enumerate}
\end{Thm}

This is significant as triangle groups do not fall into the setting of \cite{ChaoticDeloneSet} and moreover, concrete examples of chaotic Delone sets from hyperbolic cut and project schemes are limited to a small family of examples generated by a single construction (see \cite[Example\;1.5]{ChaoticDeloneSet}) and the arguments used to verify this do not, in general, apply to triangle groups. 
Moreover, we highlight that \Cref{T:MainResult}\,(3) holds in a wider generality, see \Cref{Prop:lengths_of_gaps} for further details.

\section{Preliminaries}

In this section, we discuss some necessary preliminaries. For a complete treatment of hyperbolic geometry, we direct the reader to, for instance  \cite{beardon2012geometry,EinsiedlerWard2011,Katok1992FuchsianGroups,ErgodicThryDGroups}, and for aperiodic order, we suggest \cite{Aperiodic_Order_Vol1}.

As noted in the introduction, we denote the Poincaré disc by $\mathbb{D}=\{z\in\mathbb{C}:|z|<1\}$ and equip it with the hyperbolic metric $d$ induced by the differential given in \eqref{eq:hyperbolic_differntial}. Following \cite{beardon2012geometry}, we recall that geodesics in $\mathbb{D}$ are arcs generated by intersections of $\mathbb{D}$ with circles in $\overline{\mathbb{C}}$ perpendicular to $\mathbb{S}^{1} = \{z\in\mathbb{C}:|z|=1\}$, namely the boundary of $\mathbb{D}$. Throughout, when given a parametrisation $\kappa : \mathbb{R} \to \mathbb{D}$ of a geodesic $k$, we assume that $\kappa$ has unit speed.

The group of orientation preserving isometries of $\mathbb{D}$, denoted by $\operatorname{Isom}^{+}(\mathbb{D})$, is isomorphic to the group $\operatorname{PSL}(2,\mathbb{R}) = \operatorname{SL}(2,\mathbb{R})/\{\pm I\}$, and consists of M\"obius transformations of the form
\begin{align}\label{eq:psl2element}
    \gamma(z)=\frac{az+\overline{b}}{bz+\overline{a}},
\end{align}
where $a,b\in\mathbb{C}$ and $|a|^2-|b|^2=1$. 
For ease of exposition, we also interpret such a $\gamma$ as a $2\times 2$-matrix
    \begin{align*}
        \begin{pmatrix}
            a & \overline{b}\\
            b & \overline{a}
        \end{pmatrix}.
    \end{align*}
A \textit{Fuchsian group} is a discrete subgroup $\Gamma$ of $\operatorname{Isom}^{+}(\mathbb{D})$, namely a subgroup of $\operatorname{Isom}^{+}(\mathbb{D})$ such that the orbit $\Gamma(0) = \{\gamma(0) : \gamma \in \Gamma \}$ is a discrete subset of $\mathbb{D}$. We say that $\Gamma$ is \textit{non-elementary} if $\Gamma (x)$ is not finite for any $x \in \mathbb{D}$. Moreover, a Fuchsian group is called \textit{cocompact} if its quotient space $\Gamma\backslash \mathbb{D}$ is compact.
It is also important to note that cocompact Fuchsian groups have no parabolic element; a \textit{parabolic element} of $\operatorname{Isom}^{+}(\mathbb{D})$ is an element of the form given in \eqref{eq:psl2element} with trace $2$, namely $|a + \overline{a}| = 2$.

A point $\xi\in\mathbb{S}^{1}$ is a \textit{limit point} of a Fuchsian group if for one, and hence every, point $x\in \mathbb{D}$ the orbit $\Gamma(x)$ accumulates at $\xi$. The set of all limit points of a Fuchsian group $\Gamma$ is called the \textit{limit set} of $\Gamma$ and denoted by $\Lambda(\Gamma)$. We say that $\Gamma$ as of the \textit{first kind} if $\Lambda(\Gamma)=\mathbb{S}^{1}$ and of the \textit{second kind} otherwise. 

A \textit{fundamental domain} of $\Gamma$ is a region $\tau\subseteq \mathbb{D}$ such that $\tau^{\circ} \cap \gamma(\tau^{\circ}) = \emptyset$ for all $\gamma \in \Gamma \setminus \{I\}$, where $I$ denotes the identity and $\tau^{\circ}$ is the interior of $\tau$. 
If $\Gamma$ is cocompact, then its fundamental domain $\tau$ has finite area, finitely many sides and hence finitely many vertices; thus, $\Gamma$ is finitely generated and of the first kind. If $\Gamma$ is of the first kind, then it is non-elementary; and so if $\Gamma$ is cocompact, it is non-elementary.

It is also possible to construct a fundamental domain of a cocompact Fuchsian group, where a side is paired (mapped) to another side by an element of $\Gamma$ called a side pairing element. In fact, these side pairing elements generate $\Gamma$.

\subsection{Length spectrum}

We call the quotient space $\Sigma_{\Gamma}=\mathbb{D}/\Gamma$ the \textit{hyperbolic orbifold of $\Gamma$}, and denote the unit tangent bundle of $\Sigma_{\Gamma}$ by $\operatorname{ST}(\Sigma_\Gamma)=\operatorname{ST}(\mathbb{D})/\Gamma$, observing that $\operatorname{ST}(\Sigma_\Gamma)$ is a 3-manifold. 

An element $\gamma \in \operatorname{Isom}^{+}(\mathbb{D})$, is \textit{hyperbolic} if and only if it is diagonalisable over $\mathbb{R}$; or equivalently if it has trace strictly greater than $2$. Moreover, for a non-elementary Fuchsian group $\Gamma$, the hyperbolic elements of $\Gamma$ are in one-to-one correspondence with the closed geodesics on $\Sigma_\Gamma$, namely to each hyperbolic $\gamma \in \Gamma$, there exists a closed geodesic of $\Sigma_\Gamma$ which remains invariant under the natural action of $\gamma$ on $\Sigma_\Gamma$, and vice versa.

For a hyperbolic element $\gamma\in\Gamma$, we have that the \textit{translation length} of $\gamma$, denoted $l(\gamma)$, is given by $l(\gamma)=2\log(|\lambda|)$, where $\lambda$ is the eigenvalue of $\gamma$ with $|\lambda|\geq 1$. We define the \textit{length spectrum of $\Gamma$} by 
    \begin{align*}
    L(\Gamma)
    = \{ x \in \mathbb{R} : x \; \text{is hyperbolic length of a closed geodesic on} \; \Sigma_\Gamma\}
    = \{l(\gamma) : \gamma\in\Gamma \; \text{is diagonalisable}\}.
    \end{align*}
This latter equality follows from the fact that the hyperbolic elements of $\Gamma$ are in one-to-one correspondence with the closed geodesics on $\Sigma_\Gamma$, and that if $\gamma$ is a diagonal hyperbolic element, then there exists a $\lambda \in \mathbb{R} \setminus \{ 0, 1\}$ such that 
    \begin{align*}
    \gamma(z) = \frac{(-\lambda^{2}-1)z+(-\lambda^{2}i+i)}{(\lambda^{2}i-i)z + (-\lambda^{2}-1)},
    \end{align*}
for $z \in \mathbb{D}$, in which case, the closed geodesic corresponding to $\gamma$ on $\Sigma_\Gamma$, is the lift of $\mathbb{D}$ to $\Sigma_\Gamma$ of the straight line segment connecting $0$, the origin of Poincar\'e disc, and the point $(\lambda^{2}-1)i/(\lambda^2+1) \in \mathbb{D}$, whose hyperbolic length is equal to $2\log(\lvert \lambda \rvert)$. Note, this straight line segment lies completely on a geodesic of $\mathbb{D}$, namely a diameter of $\mathbb{D}$.

\subsection{The geodesic flow on the unit tangent bundle}\label{sec:GF}

For $(x,v) \in \operatorname{ST}(\mathbb{D})$ we let $\kappa_{(x, v)}$ denote the parameterisation of the unique geodesic with $\kappa_{(x,v)}(0) = x$ and $\kappa_{(x,v)}'(x)=v$. We also recall that parameterisations of geodesics have unit speed. The \textit{geodesic flow }$\varphi_t:\operatorname{ST}(\mathbb{D})\rightarrow \operatorname{ST}(\mathbb{D})$ on $\mathbb{D}$ is a flow on the unit tangent bundle $\operatorname{ST}(\mathbb{D})$ defined by 
\begin{align*}
    \varphi_t(x,v) = (\kappa_{(x,v)}(t),\kappa'_{(x,v)}(t)),
\end{align*}
for $(x, v) \in \operatorname{ST}(\mathbb{D})$ and $t \in \mathbb{R}$. Moreover, for an isometry $\gamma$ of the Poincaré disc, the action of $\gamma$ on $\operatorname{ST}(\mathbb{D})$ is given by $\gamma(x,v)=(\gamma(x),\gamma'(x)v)$. Note this action is well defined since isometries of the Poincaré disc are conformal, and thus preserve angles.  Moreover, this yields that $\varphi_{t} \circ \gamma = \gamma \circ \varphi_{t}$. From this, for a Fuchsian group $\Gamma$, we define the geodesic flow on $\Sigma_{\Gamma}$ to be the lift of $\varphi_t$ from $\operatorname{ST}(\mathbb{D})$ to $\operatorname{ST}(\Sigma_\Gamma)$.

We say that the geodesic flow of $\Gamma$ is \textit{topologically transitive} if there exists a $(x, v) \in \operatorname{ST}(\Sigma_\Gamma)$ with $\{ \varphi_{t}(x,v) : t \in \mathbb{R} \}$ dense in $\operatorname{ST}(\Sigma_\Gamma)$.

\begin{Lem}\label{Lem:trans anosov flow}
If $\Gamma$ is a cocompact, then the geodesic flow of $\Gamma$ is topologically transitive. Moreover, the set of the periodic points of the geodesic flow is dense in $\operatorname{ST}(\Sigma_\Gamma)$.
\end{Lem}

\begin{proof}
As $\Gamma$ is cocompact, it is finitely generated, and so Selberg's lemma \cite[Page 10]{beardon2012geometry} yields that there exists a torsion-free normal subgroup $\Gamma_1\subset\Gamma$ of finite index.  Further, by \cite[Corollary 4.2.3.]{Katok1992FuchsianGroups}, its fundamental domain is also compact. Applying \cite[Theorem 3.1.2]{Katok1992FuchsianGroups}, as $\Gamma_{1}$ is a finite index normal subgroup of $\Gamma$, it has a compact fundamental domain. By a second application of \cite[Corollary 4.2.3.]{Katok1992FuchsianGroups}, the quotient space $M_1=\Gamma_1\backslash \mathbb{D}$ is also compact. Hence, \cite[Theorem 5.4.15]{KatokHasselblatt1995} gives the geodesic flow of $\Gamma_{1}$ is topologically transitive. Letting $n \in \mathbb{N}$ denote the index $[\Gamma : \Gamma_{1}]$, the natural projection
    \begin{align*}
        \widetilde{\pi} : \operatorname{ST}(\Sigma_{\Gamma_1}) \to \operatorname{ST}(\Sigma_\Gamma)
    \end{align*}
is surjective and $n$-to-$1$. Denoting the geodesic flow of $\Gamma$ by
$\varphi_t$ and the geodesic flow of $\Gamma_{1}$ by $\psi_t$, it follows that $\widetilde{\pi}\circ \psi_t=\varphi_t\circ\widetilde{\pi}$ for all $t\in\mathbb{R}$. Combining this with the fact that $\widetilde{\pi}$ is $n$-to-$1$ and $\psi_t$ is topologically transitive, yields $\varphi_{t}$ is topologically transitive. The second statement on the periodic points follows in the same way utilising \cite[Theorem 5.4.14]{KatokHasselblatt1995} in place of \cite[Theorem 5.4.15]{KatokHasselblatt1995}.
\end{proof}

\subsection{Cocompact Fuchsian triangle groups}\label{cocompactTriGp}

Given a cocompact Fuchsian group $\Gamma$, we call the orbit of a vertex $v$ of $\tau$, the fundamental domain of $\Gamma$ a \textit{cycle}.  Letting $C$ denote cycle, for each $w \in C$, we let $G_{w} \subseteq \Gamma$ denote the stabilizer of $w$. If $G_{w}$ is non-empty, we set $\operatorname{Ord}(G_{w})$ to be the common order of $G_{w}$, and otherwise we set $\operatorname{Ord}(G_{w}) = 1$.  Since for $u, w \in C$ we have that $G_{w}$ and $G_{u}$ are conjugate, we have  $\operatorname{Ord}(G_{w}) = \operatorname{Ord}(G_{u})$, and thus we denote this common value by $\operatorname{Ord}(C)$.  In the case when $\operatorname{Ord}(C) = 1$, we call $C$ an \textit{accidental cycle}.  Note, since $\Gamma$ is discrete and cocompact, $\operatorname{Ord}(C)$ is always finite. Moreover, if $\gamma \in \operatorname{Isom}^{+}(\mathbb{D})$ fixes a $z \in \mathbb{D}$ it is called \textit{elliptic}; this is equivalent to $\gamma$ having trace strictly less than $2$. Thus, if a cycle $C$ of $\Gamma$ is not an accidental cycle, then $G_{w}$, for $w \in C$, is a cyclic group generated by an elliptic element of finite order, hence we call $C$ an \textit{elliptic cycle}.

Let $v$ be a vertex of $\tau$, and let $\theta_{v}$ represent the internal angle of $\tau$ that subtends at $v$. If \mbox{$C = \{ v_{1}, \cdots, v_{n} \}$} denotes the cycle containing $v$, the \textit{angle sum} $\theta(C)$ of $C$ is defined as $\theta_{v_{1}} + \theta_{v_2}+\cdots+ \theta_{v_n}$. If $\tau$ is convex, then $\theta(C)=2\pi/\operatorname{Ord}(C)$, see \cite[Theorem 9.3.5.]{beardon2012geometry}. Thus, $\theta(C)=2\pi$ for an accidental cycle $C$.

Let $\mathcal{C}$ denote the collection of elliptic cycles, let $C_{1}, \cdots, C_{r}$ be an enumeration of the elements of $\mathcal{C}$, and set $m_j = \operatorname{Ord}(C_{j})$, for $j \in \{ 1, \cdots, r \}$. The symbol $(g; m_1,m_2,\cdots,m_r)$ is called the \textit{signature} of $\Gamma$, where $g$ denote the genus of the surface $\Gamma\backslash\mathbb{D}$.

A cocompact Fuchsian group with the signature $(0; m_1,m_2,m_3)$, where 
    \begin{align*}
    \frac{1}{m_1}+\frac{1}{m_2}+\frac{1}{m_3}<1,
    \end{align*}
and group presentation $\langle A, B : A^{m_{1}} = B^{m_{2}} = (AB)^{m_3} = I \rangle$, is called a cocompact Fuchsian \textit{triangle group}. Since, a cocompact Fuchsian triangle group always has genus zero, we abbreviate its signature to $(m_1,m_2,m_3)$. We also note, given a Fuchsian triangle group, by \cite[Theorem 10.6.4]{beardon2012geometry}, it exhibits a convex fundamental domain which is necessarily a hexagon or a quadrilateral.

For the following sections (\Cref{sec:QFD,sec:HFD}), we will use the notation from \cite{schmidt2024continuous} and throughout these sections $\Gamma$ will denote a cocompact Fuchsian triangle group with convex fundamental domain which we will denote by $\mathcal{F}$.

\subsubsection{Quadrilateral fundamental domain}\label{sec:QFD}

We label the vertices of $\mathcal{F}$ as $v_1$, $v_{2}$, $v_{3}$, $v_{4}$, and to standardise, we assume they have internal angles of $\pi/m_3$, $2\pi/m_2$, $\pi/m_3$ and $2\pi/m_1$, respectively. See \Cref{f:fd663} for an example.

For $i \in \{ 1, 2, 3, 4 \}$, we let $r_i$ denote the side of $\mathcal{F}$ meeting $v_i$ at its right (facing inward). We label the generators of $\Gamma$ that are the side pairings of $\mathcal{F}$, taking $r_i$ to its paired side, by $T_i$.

\begin{figure}[ht]
\includegraphics[scale=0.35]{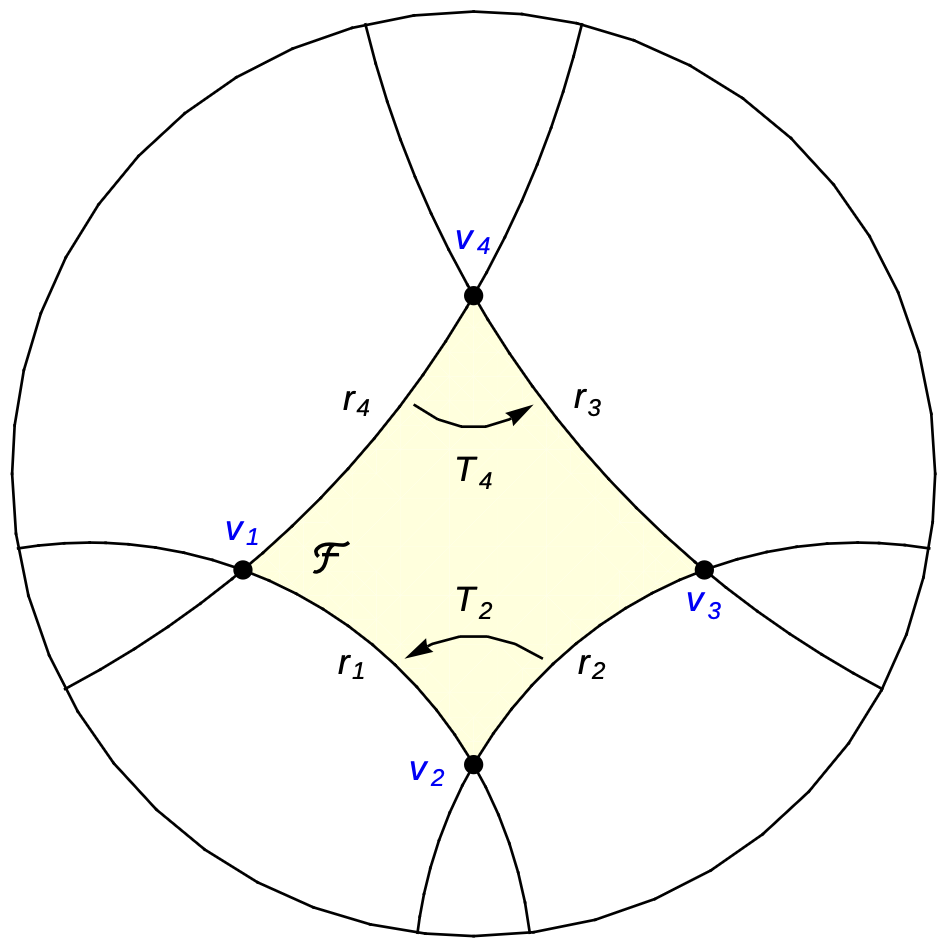} 
\caption{A quadrilateral fundamental domain for the triangle group with signature (6,6,3).}
\label{f:fd663}
\end{figure}

Since for a cocompact Fuchsian triangle group its generators are of finite order, and hence elliptic, for $i \in \{1, 2, 3, 4\}$, letting
    \begin{align}\label{eq:QuadDomainPair}
    \sigma(i) = \begin{cases}
    i-1     &   \text{if} \; i \; \text{is even}, \\
    i+1     &   \text{if} \; i \; \text{is odd},
    \end{cases}
    \qquad \text{and} \qquad
    \rho(i) = 
    \begin{cases}
    1               & \text{if} \; \sigma(i)=4,\\
    \sigma(i)+1     & \text{otherwise},
    \end{cases}
    \end{align}
we may assume without loss of generality, that $T_i(r_i)=r_{\sigma(i)}$ and $T_i(v_i)=v_{\rho(i)}$.   Accordingly, for $i \in \{ 1, 2, 3, 4\}$ we have $T_{\sigma(i)}T_i=I$ and $T_{i-1}(v_i)=v_{\rho(i)}$. Thus, we have three elliptic cycles $\{v_1, v_3\}$, $\{v_{1}\}$ and $\{ v_{4}\}$. Moreover, we find that $T_2 T_4$ and $T_3T_1$ fix $v_1$ and rotate through an angle of  $2\pi/m_3$ in opposite directions. Similarly, we have $T_4T_2$ and $T_1T_3$ fix $v_3$ and rotate by the angle of $2\pi/m_3$ in opposite directions.  
  
\subsubsection{Hexagonal fundamental domain}\label{sec:HFD}

We begin by observing that if a cocompact Fuchsian triangle group exhibits a hexagonal fundamental domain, it necessarily exhibits a quadrilateral fundamental domain; however the converse is not necessarily true, see for instance \cite[Theorem 10.6.4]{beardon2012geometry}. The proof of \cite[Theorem 10.6.4]{beardon2012geometry} also demonstrates that, in this former setting, namely when a cocompact Fuchsian triangle group $\Gamma$ exhibits a hexagonal fundamental domain $\mathcal{F}$, then it has four cycles, one accidental cycle and three elliptic cycles.  We therefore label the vertices of $\mathcal{F}$ as $v_1,v_2,v_3,a_1,a_2,a_3$, where each $v_{i}$ forms an elliptic cycle and the $a_i$'s form an accidental cycle. The vertices $v_i$ each have an internal angle of $2\pi/{m_i}$. Denoting the internal angles at the vertices $a_i$ by $\theta_i$, we necessarily have that $\theta_1+\theta_2+\theta_3=2\pi$. Let $s_i$ denote the sides of $\mathcal{F}$ corresponding to the side to the left of $v_i$ (facing inward), whereas $t_i$ corresponds to the sides to the right of $v_i$ (facing inward). The side pairings $U_i\in\Gamma$ pair the side $t_i$ to the side $s_i$, namely, they are rotations at $v_i$ by the angle of $2\pi/{m_i}$, and $U_{1}(a_{1}) = a_{3}$, $U_{2}(a_{2}) = a_{1}$ and $U_{3}(a_{2}) = a_{3}$. See \Cref{f:hexagonFD} for an example. 

\begin{figure}[ht]
\includegraphics[scale=0.35]{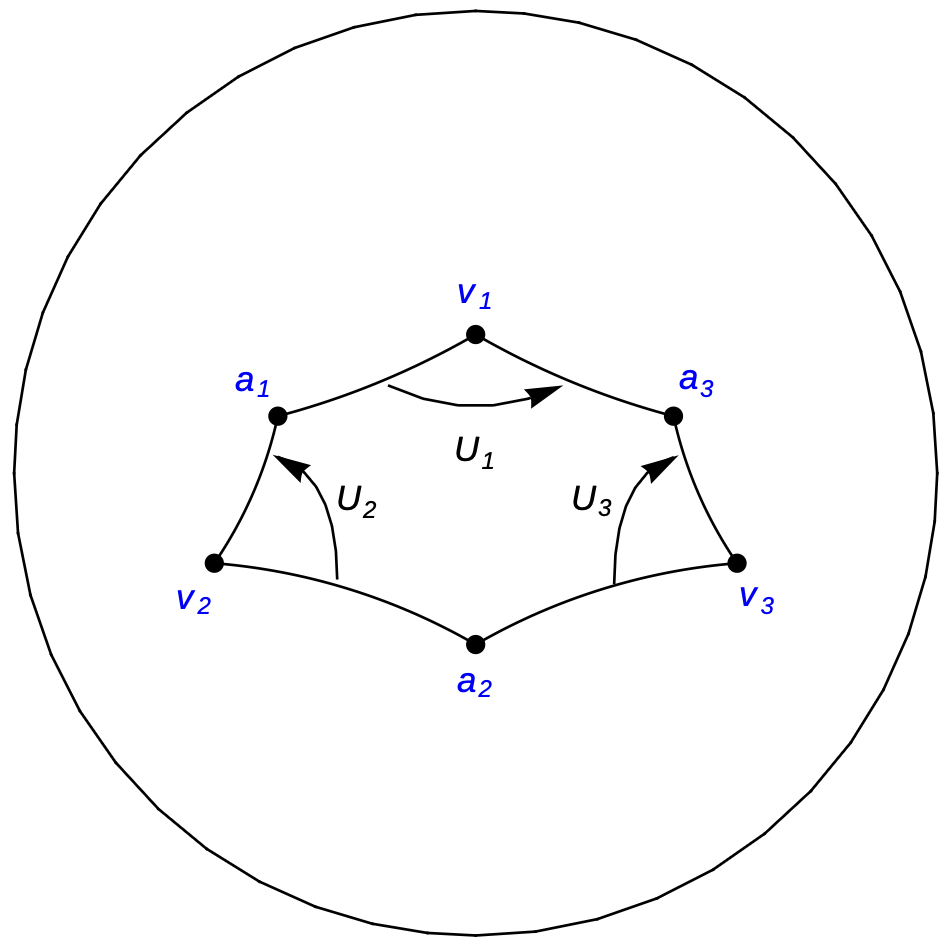} 
\caption{A hexagonal fundamental domain for the triangle group with signature $(6, 6, 3)$ and the side pairing elements $U_i$.}
\label{f:hexagonFD}
\end{figure}

\subsection{Chaotic Delone sets}

Delone sets have had a long history of study with respect to dynamics and aperiodic order. Chaotic Delone sets, introduced in \cite{ChaoticDeloneSet}, are an interesting subset of Delone sets and are defined as follows:

\begin{Def}\label{def:CD}
    Let $\varepsilon$ and $\delta$ denote two fixed positive real numbers.
    A subset $S$ of a metric space $(X,d)$ is \textit{$(\varepsilon,\delta)$-Delone} if the following two conditions hold.
    \begin{itemize}
        \item $\varepsilon$-relatively dense: For every $x\in X$ there exists $y \in X$ with $d(x,y)\leq\varepsilon$.
        \item $\delta$-separated: For distinct $x,y\in X$, we have that $d(x,y)\geq \delta$.
    \end{itemize}
    For $n \in \mathbb{N}$, we let $\operatorname{Del}_{\varepsilon,\delta}(n)$ denote the set of $(\varepsilon,\delta)$-Delone subsets of $\mathbb{R}^n$, and in the case when $n=1$, for ease of notation we will write $\operatorname{Del}_{\varepsilon,\delta}$.
\end{Def}

The set $\operatorname{Del}_{\varepsilon,\delta}(n)$ has a canonical, compact, metrisable topology, introduced in \cite{LocalRubberTop}, see also \cite[Page 9]{BAAKE_LENZ_2004}, such that the action of $\mathbb{R}^n$ given by 
\begin{align*}
    \mathbb{R}^n\times \operatorname{Del}_{\varepsilon,\delta}(n)&\rightarrow\operatorname{Del}_{\varepsilon,\delta}(n)\\
    (v,S)&\rightarrow S-v \, = \{ s -v : s \in S\} 
\end{align*}
is continuous. For $S\in\operatorname{Del}_{\varepsilon,\delta}(n)$ we write $[S]$ for the orbit $S+\mathbb{R}^n$ and let $\overline{[S]}$ denote the closure of this set in the local rubber topology. A Delone set is called \textit{periodic} if the orbit $[S]$ is compact.  We can now define what it means for a Delone set to be chaotic.

\begin{Def}\cite[Definition 1.2]{ChaoticDeloneSet}\label{def:ChaoticDeloneSet}
    A Delone set $S$ is \textit{almost chaotic} if the union of periodic orbits is dense in $\overline{[S]}$. The set $S$ is \textit{chaotic} if it is almost chaotic and aperiodic.  Recall, that we say that $S$ is \textit{aperiodic} if $S-v\neq S$ for all $v\in \mathbb{R}^n\setminus \{0\}$.
\end{Def}

For $n \in \mathbb{N}$ and $r$ a positive real number, we define
    \begin{align*}
    N_r(n)=\big\{(S,S')\in (\operatorname{Del}_{\varepsilon,\delta}(n))^{2} : S\cap B(0,r)\subseteq S'+B(0,1/r) \; \text{and} \; S'\cap B(0,r)\subseteq S+B(0,1/r)\big\},
    \end{align*}
where, for $x\in\mathbb{R}^n$ and $\eta$ a positive real number, $B(x,\eta) = \{ y : \lvert y-x \rvert < \eta \}$, and for a given set $K$, we let $K + B(0,\eta)$ denote the set $\eta$-parallel neighbourhood of $K$, that is, the set $\{ k + t : k \in K \; \text{and} \; t \in B(0,\eta) \}$. In the case when $n = 1$, for ease of notation, we write $N_{r}$ for $N_{r}(n)$. Moreover, for an arbitrary metric space, we use the same notation for an open ball centred at a given point $x$ of the metric space with radius $\eta$, and $\overline{B}(x,\eta)$ to denote the closure of $B(x,\eta)$.

With these definitions we can provide the following characterisation. This lemma is useful in verifying when a Delone set is almost chaotic.

\begin{Lem}[{\cite[Lemma 2.1]{ChaoticDeloneSet}}]\label{Lem: Chaotic Delone rework}
    An $(\epsilon,\delta)$-Delone set $S \subset \mathbb{R}^{n}$ is almost chaotic if and only if, for every $r\in\mathbb{N}$, there is a periodic Delone set $S'\in\operatorname{Del}_{\epsilon,\delta}(n)$ such that $(S,S')\in N_r(n)$, and, for any $s\in\mathbb{N}$ there exists a point $x\in\mathbb{R}^n$ such that $(S-x,S')\in N_s(n)$.
\end{Lem}

\section{From hyperbolic polygons to chaotic Delone sets} \label{sec:Chaotic_Delone}

Throughout this section let $\Gamma$ be a cocompact Fuchsian group with fundamental domain $\tau$ a hyperbolic polygon, and let \mbox{$T_{1}, \cdots, T_{n}$} denote the generators of $\Gamma$ which are side pairings of $\tau$. Also assume that none of these generators are their own inverse. By \Cref{Lem:trans anosov flow}, the geodesic flow on $\operatorname{ST}(\Sigma_\Gamma)$ is transitive, and thus, there exists $(x, v) \in \operatorname{ST}(\Sigma_{\Gamma})$ with dense orbit under the flow. Let us fix such a point $(x_{0}, v_{0}) \in \operatorname{ST}(\Sigma_{\Gamma})$ and also fix $\ell$ to be the unique geodesic in $\mathbb{D}$ given by the parametrisation $\kappa_{(\widetilde{x}_{0}, \widetilde{v}_{0})}$, where $(\widetilde{x}_{0}, \widetilde{v}_{0}) \in \operatorname{ST}(\mathbb{D})$ is the unique point which lifts to $(x_{0}, v_{0}) \in \operatorname{ST}(\Sigma_{\Gamma})$, see \Cref{sec:GF}.  Note the uniqueness follows from the fact that the geodesic flow along $\ell$ is dense in $\operatorname{ST}(\Sigma_{\Gamma})$.  With this in mind, and for ease of exposition, we abuse notation and write 
$\kappa_{(x_{0}, v_{0})}$ in replace of $\kappa_{(\widetilde{x}_{0}, \widetilde{v}_{0})}$. We also, fix $x \in \tau^{\circ}$ and $\rho$ a positive real number.

To obtain our main result, \Cref{Thrm-Delone if and only if one sided tangency}, we adapt \cite[Lemmas 4.4, 4.8, 4.9 and 4.10]{ChaoticDeloneSet} to our setting.
The main difference between our proof, and those of \cite{ChaoticDeloneSet}, is that we do not consider two-sided tangency, but the readily verifiable condition that all geodesics in the Poincar\'e disc intersect $\Gamma (B(x,\rho))$.
Indeed, as we show in \Cref{l:two_odd_signatures_QuadDomain} and \Cref{l:hexDomain}, triangle groups have this latter property.
We include the statements since they give a strong understanding of the underlying geometry in play. Additionally we note that their proofs follow using the same techniques as in \cite{ChaoticDeloneSet} and thus are consigned to the appendix.

\begin{Lem}[{Analogue to \cite[Lemma 4.4]{ChaoticDeloneSet}}]\label{lem:Chaotic Behaviour of geodesic}
    Let $r \in \mathbb{R}$ positive be given.
    \begin{enumerate}
        \item There is a geodesic $k$ in $\mathbb{D}$, lifting to a closed geodesic on $\Sigma_\Gamma$ with $(S^+(\ell,\rho,x),S^+(k,\rho,x))\in N_r$.
        \item For an arbitrary geodesic $k$ in $\mathbb{D}$, there exists $a\in\mathbb{R}$ such that 
        \begin{align*}
        (S^+(\ell,\rho,x) -a,S^+(k,\rho,x))\in N_r.
        \end{align*}
    \end{enumerate}
\end{Lem}

\begin{Lem}[{Analogue to \cite[Lemma 4.8]{ChaoticDeloneSet}}]\label{lem: ap} \mbox{ }
        \begin{enumerate}
            \item If $\rho<\operatorname{inj}(\Gamma,x)$, then $S^+(\ell,\rho,x)$ is $\delta$-separated, where $\delta=2(\operatorname{inj}(\Gamma,x)-\rho)$.
            \item If $\operatorname{inj}(\Gamma,x) \leq \rho$, then $S^+(\ell,\rho,x)$ is not $\delta$-separated for all $\delta\in\mathbb{R}$ positive.
        \end{enumerate}
\end{Lem}

\begin{Lem}[{Analogue to \cite[Lemma 4.9]{ChaoticDeloneSet}}]\label{lem: appp}
If all geodesics on $\Sigma_\Gamma$ intersect $B(\pi(x),\rho)$, then the set $S^+(\ell,\rho,x)$ is $\varepsilon$-relatively dense for some real $\varepsilon>0$.
\end{Lem}

\begin{Lem}[{Analogue to \cite[Lemma 4.10]{ChaoticDeloneSet}}]\label{lem: apppp}
    If all geodesics on $\Sigma_\Gamma$ intersect $B(\pi(x),\rho)$, then the set $S^+(\ell,\rho,x)$ is aperiodic.
\end{Lem}

By \Cref{lem:Chaotic Behaviour of geodesic,lem: ap,lem: appp,lem: apppp} in tandem with the characterisation of an almost chaotic Delone set in \Cref{Lem: Chaotic Delone rework} we obtain the following theorem. 
Note, \Cref{lem: ap}\,(2), demonstrates that this result is optimal in the sense that $\rho$ can not be chosen greater than or equal to $\operatorname{inj}(\Gamma,x)$.

\begin{Thm}\label{Thrm-Delone if and only if one sided tangency}
    If all geodesics on $\Sigma_\Gamma$ intersect $B(\pi(x),\rho)$ and $\rho<\operatorname{inj}(\Gamma,x)$, then $S^+(\ell,\rho,x)$ is Delone. Moreover, if $S^+(\ell,\rho,x)$ is Delone, then it is chaotic.
\end{Thm}

\begin{proof}
    By \Cref{lem: appp}\, and \Cref{lem: ap}\, (1), if all geodesics on $\Sigma_\Gamma$ intersect $B(\pi(x),\rho)$ and $\rho<\operatorname{inj}(\Gamma,x)$, then $S^+(\ell,\rho,x)$ is Delone for some $\varepsilon,\delta \in \mathbb{R}$. Therefore, \Cref{lem:Chaotic Behaviour of geodesic}\,(2), implies for any geodesic $k$, lifting to a closed geodesic on $\Sigma_\Gamma$, the set $S^+(k,\rho,x)$ is also $(\varepsilon,\delta)$-Delone. Further, since $k$ lifts to a closed geodesic, we observe that  $S^+(k,\rho,x)$ is additionally periodic Delone. Using the characterisation of an almost chaotic Delone set in \Cref{Lem: Chaotic Delone rework}, a direct application of \Cref{lem:Chaotic Behaviour of geodesic} and \Cref{lem: apppp} yields the required result.
\end{proof}

It is natural to consider a point $x\in\tau$ equidistant to all sides of the polygon $\tau$, if such a point exists. When making this restriction, it can be shown that there exists $\rho>0$ such that $(x,\rho)$ satisfies the conditions of \Cref{Thrm-Delone if and only if one sided tangency} dependant only on a simple to verify condition on the sides of the polygon, circumventing Condition B. Hence, showing that for the setting given, the cut and project set $S^+(\ell,x,\rho)$ is chaotic Delone. This is precisely our main result \Cref{Nice condition for aperiodic delone sets}. Before, stating \Cref{Nice condition for aperiodic delone sets}, we require one further definition.  If $m$ is a side of $\tau$, we call the geodesic containing $m$, an \textit{extended side} and denote it by $\mathfrak{m}$.

\begin{Thm}\label{Nice condition for aperiodic delone sets}
    Suppose there exists a point $y\in \tau^{\circ}$ such that the closed ball $\overline{B}(y,\mu_y)$ is tangent to all sides of $\tau$, where for ease of notation $\mu_{y} = \operatorname{inj}(\Gamma,y)$. There exists $\rho<\mu_y$ such that $S^+(\ell,\rho,y)$ is chaotic Delone if and only if all extended sides of the polygon $\tau$ intersect $\Gamma(\tau^\circ)$.
\end{Thm}

\begin{proof}
For the forward direction, if an extension $\mathfrak{m}$ of a side $m$ of $\tau$ does not intersect $\Gamma(\tau^\circ)$, then for any $\rho \in (0,\mu_y)$ the geodesic $\mathfrak{m}$ does not intersect $\Gamma(\overline{B}(y,\rho))\subseteq \Gamma(\tau^\circ)$ hence the set $S^+(\mathfrak{m},\rho,y)$ is empty.
Hence, by \Cref{lem: ap}\,(1), the set $S^+(\ell,\rho,y)$ has arbitrarily large empty patches and thus is not Delone for any $\rho<\mu_y$.

To prove the reverse direction we will verify that there exists a $\rho$, sufficiently close to $\mu_y$, such that all geodesics intersect the region $\Gamma(B(y,\rho))$. This being sufficient to verify the conditions in \Cref{Thrm-Delone if and only if one sided tangency} and hence that $S^+(\ell,\rho,y)$ is chaotic Delone.

Consider an arbitrary side $m_i$ and let $\mathfrak{m}_i$ denote the extended side of $m_{i}$, where $i\in\{1,\cdots,2n\}$. We parametrise $\mathfrak{m}_i$ with unit speed and such that $0$ is mapped to the singleton $m_i \cap \overline{B}(y,\mu_y)$.  We will abuse notation and use the symbol $\mathfrak{m}_{i}$ to denote the extended side of $m_{i}$ as well as its parametrisation. Moreover, we assume that the parametrisation is such that the normal direction of $m_i$ is directed towards the interior of the polygon $\tau$. Observe that since $\mathfrak{m}_i$ intersects  $\Gamma(\tau^\circ)$, it intersects $\Gamma(B(y,\mu_y))$. Indeed, this is the case since, by assumption, there must be two consecutive closed intervals $I_1, I_2 \subset \mathbb{R}$ with pairwise disjoint interiors, where $\mathfrak{m}_i(I_1)$ is a side of an image of $\tau$, say $\gamma_1(\tau)$ for some $\gamma_{1} \in \Gamma$, and $\mathfrak{m}_i(I_2)$ intersects an image of $\tau^\circ$, say $\gamma_2(\tau^\circ)$ for some $\gamma_{2} \in \Gamma$. Therefore, we have $\mathfrak{m}_{i}(I_1\cap I_2)$ is a vertex on $\gamma_1(\tau)\cap\gamma_2(\tau)$ and since $\mathfrak{m}_i(I_2)$ intersects $\gamma_2(\tau^\circ)$, and $\gamma_2(\overline{B}(y,\mu_y))$ is tangent to the sides of $\gamma_2(\tau)$, the geodesic $\mathfrak{m}_i$ must intersect $\gamma_2(B(y,\mu_y))$.

Thus, there exists $\gamma_{i} \in \Gamma$ such that $\mathfrak{m}_i$ intersects $\gamma_i(B(y,\mu_y))$ for some $\gamma_i\in\Gamma$. Therefore, there exist $s_i, \epsilon_i\in\mathbb{R}$ with $\epsilon_i$ positive, such that $d(\mathfrak{m}_i(s_i),\gamma_i(y))+\epsilon_i<\mu_y$. Setting $\epsilon=\min_i\{\epsilon_i\}$, noting since $\tau$ has finitely many sides, $\epsilon>0$, there exists $\delta \in \mathbb{R}$ such that for any $i \in\{1,\cdots,2n\}$, the inequality $d(\mathfrak{m}(s_i), \gamma_i(y))+\epsilon< \delta<\mu_y$ holds. Hence, 
    \begin{align*}
    B(\mathfrak{m}_i(s_i),\epsilon)\subset \overline{B}(\mathfrak{m}_i(s_i),\epsilon)\subseteq B(\gamma_i(y),\delta)=\gamma_i(B(y,\delta)).
    \end{align*}
For $i\in\{1,\cdots,2n\}$, let $\mathfrak{t}_{i,s_i}$ be the unique geodesic perpendicular to $\mathfrak{m}_i$ with $\mathfrak{t}_{i,s_i}(0) = \mathfrak{m}_i(s_i)$, parametrised with unit speed and such that $\mathfrak{t}_{i,s_i}'(0)$ is equal to the the normal of $\mathfrak{m}_i$ at $\mathfrak{m}_i(s_i)$. 
We define $\mathfrak{q}_i^\pm$ to be the unique geodesic passing through the points $\mathfrak{m}_i(0)$ and $\mathfrak{t}_{i,s_i}(\pm\epsilon)$. By construction, $\mathfrak{t}_{i,s_i}(\epsilon)$ and $\mathfrak{t}_{i,s_i}(-\epsilon)$ are contained in $\overline{B}(m_i(s_i),\epsilon)\subset B(\gamma_i(y),\delta)$. Therefore, $\mathfrak{q_i}^+$ and $\mathfrak{q}_i^-$ both intersect the ball $\overline{B}(\mathfrak{m}_i(s_i),\epsilon)\subset B(\gamma_i(y),\delta)$. We parametrise $\mathfrak{q}^+_i$ and $\mathfrak{q}^-_i$ such that $\mathfrak{q}^+_i(0)=\mathfrak{q}^-_i(0)=\mathfrak{m}_i(0)$ and such that their forward directions intersect $\overline{B}(\mathfrak{m}_i(s_i),\epsilon)$, see \Cref{f:boundinglines}.

\begin{figure}
\includegraphics[scale=0.7]{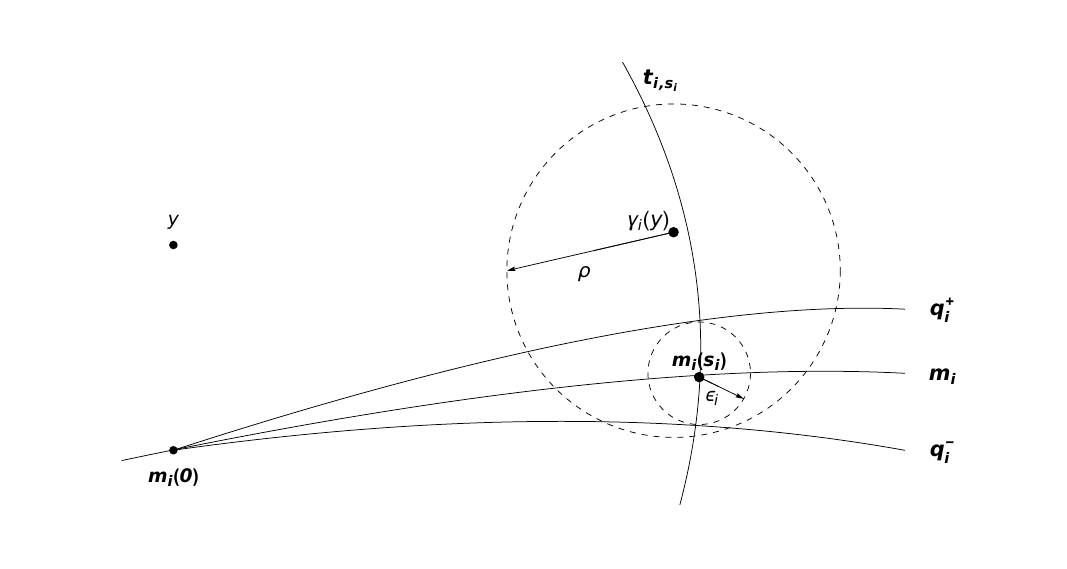} 
\caption{Sections of the geodesics $\mathfrak{q}_i^{\pm}$ passing through $\mathfrak{t}_{i,s_i}(\pm\epsilon)$ and intersecting the ball $\overline{B}(\mathfrak{m}_i(s_i),\epsilon)\subset B(\gamma_i(y),\rho)$.}
\label{f:boundinglines}
\end{figure}

For each $i\in\{1,\cdots,2n\}$, observe that $\mathfrak{q}_i^+$ makes a positive angle $\alpha_i^+$ with $\mathfrak{m}_i$ at $\mathfrak{m}_i(0)$ and $\mathfrak{q}_i^-$ makes a negative angle $\alpha_i^-$ with $\mathfrak{m}_i$ at $\mathfrak{m}_i(0)$. Letting $y_i$ denote the centre of the image of $\tau$ sharing the side $m_i$ with $\tau$, we have that $B(y_i,\mu_y)$ and $B(y,\mu_y)$ are both tangent to $m_i$. Therefore, since $\alpha_i^+$ and $\alpha_i^-$ are non-zero, there exists $\eta_i\in (0,\mu_y)$ such that $\mathfrak{q}_i^+([0,\infty))$ and $\mathfrak{q}_i^-((-\infty,0])$ intersect $B(y,\eta_i)$ and $\mathfrak{q}_i^-([0,\infty))$ and $\mathfrak{q}_i^+((-\infty,0])$ intersect $B(y_i,\eta_i)$. Setting $\eta=\max_{i}\{\eta_i\}$, the segments $\mathfrak{q}_i^+([0,\infty))$ and $\mathfrak{q}_i^-((-\infty,0])$ intersect $B(y,\eta)$, and the segments $\mathfrak{q}_i^-([0,\infty))$ and $\mathfrak{q}_i^+((-\infty,0])$ intersect $B(y_i,\eta)$, see \Cref{f:balltraj}.

We set $\rho=\max\{\delta,\eta\}<\mu_y$ and claim that all geodesics intersect $\Gamma (B(y,\rho))$. To see this claim, let $\mathfrak{k}$ be an arbitrary geodesic and assume, without loss of generality, that $\mathfrak{k}$ intersects $\tau$. For each $i\in\{1,\cdots,2n\}$, we set $a_i^+, a_i^- \in \mathbb{R}$ be such that $\mathfrak{q}_i^+(a_i^+),\mathfrak{q}_i^-(a_i^-)\in B(y,\rho)$, see \Cref{f:balltraj}.

We set $T=\cup_i(\mathfrak{t}_{i,0}([0,\mu_y])$, where $\mathfrak{t}_{i,0}$ is defined analogously to $\mathfrak{t}_{i,s_i}$. Observe that, for $\xi_1,\xi_2\in\Gamma$ such that $\xi_1(\tau)$ and $\xi_2(\tau)$ share a side $m$, letting $z$ denote the singleton 
    \begin{align*}
    m\cap \overline{B}(\xi_1(y),\mu_y)=m\cap\overline{B}(\xi_2(y),\mu_y),
    \end{align*}
$\xi_1(T)$ connects $\xi_1(y)$ to $z$, and $\xi_2(T)$ connects $\xi_2(y)$ to $z$. 
Thus, $\xi_1(T)\cup \xi_2(T)$ connects $\xi_1(y)$ to $\xi_2(y)$.

Let $G=\{g\in\Gamma\setminus \{I\}:g(\tau)\cap\tau\neq\emptyset\}$ and observe that by construction the set $\cup_{g\in G}\,g(T)$ contains a simple loop encompassing $\tau$. Since $\mathfrak{k}$ intersects $\tau$, it intersects $\cup_{g\in G}\,g(T)$, and hence, by translating $\mathfrak{k}$ by an element of $\Gamma$, we may assume, without loss of generality that, $\mathfrak{k}$ intersects $T$. Therefore, if we can verify that any geodesic intersecting $T$ intersects $\Gamma (B(y,\rho))$ our claim follows.

Since $\mathfrak{k}$ intersects $T$, we have that $\mathfrak{k}$ intersects $\mathfrak{t}_{i,0}([0,\mu_y])$ for some $i\in\{1,\cdots,2n\}$.
We may assume that $\mathfrak{k}$ intersect $\mathfrak{t}_{i,0}([0,\mu_y-\rho])$, otherwise the claim follows trivially.
Let $\mathfrak{w}$ be the geodesic intersecting $\mathfrak{k}\cap\mathfrak{t}_{i,0}([0,\mu_y-\rho])$ and $\mathfrak{q}_i^-(a^-)$. We parametrise $\mathfrak{k}$ such that $\mathfrak{k}(0)=\mathfrak{k}\cap\mathfrak{t}_{i,0}([0,\mu_y-\rho])$ and $\mathfrak{k}((0,\infty))$ is contained in the connected component $R$ of $\mathbb{D}\setminus \mathfrak{w}$ containing $y$. Let $R_1 \subset R$ be the region contained between $\mathfrak{w}$ and the geodesic passing through $\mathfrak{k}(0)$ and $\mathfrak{q}_i^+(a^+)$, as well as not containing any non-trivial segment of $t_{i,0}$. Moreover, let $R_2 = R\setminus R_1$. See \Cref{f:balltraj} for an illustration.

\begin{figure}[th!]
\includegraphics[scale=0.35]{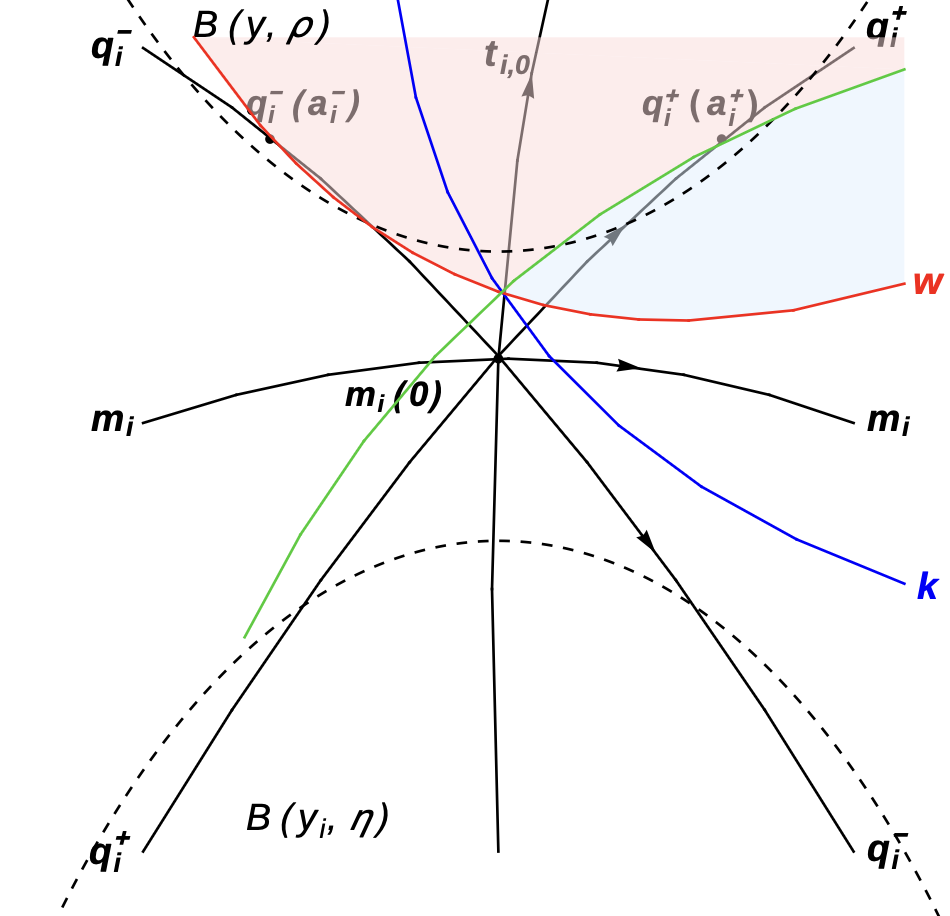} 
\caption{\textls[-10]{Examples of geodesic segments $\mathfrak{q}_i^+([0,\infty))$ and $\mathfrak{q}_i^-((-\infty,0])$ intersecting $B(y,\rho)$ at the points $\mathfrak{q}_i^+(a^+)$ and $\mathfrak{q}_i^-(a^-)$, and the geodesic segments $\mathfrak{q}_i^-([0,\infty))$ and $\mathfrak{q}_i^+((-\infty,0])$ intersecting $B(y_i,\eta)$, together with the geodesic $\mathfrak{k}$, in {\color{blue}blue}, which intersects the fundamental domain $\tau$.  Also included in {\color{Red}red}, is a segment of the geodesic $\mathfrak{w}$ passing through the points $\mathfrak{q}_i^-(a^-)$ and $\mathfrak{k}(0)$ and, in {\color{Green}green}, the geodesic passing through $\mathfrak{q}_i^+(a^+)$ and $\mathfrak{k}(0)$.  We also highlight the regions $R_1$ and $R_2$ shaded in {\color{blue}blue} and {\color{Red}red} respectively.}}
\label{f:balltraj}
\end{figure}

To complete the proof of our claim we consider the following two cases: when $\mathfrak{k}((0,\infty)) \subset R_2$, and when $\mathfrak{k}((0,\infty)) \subset R_1$. If $\mathfrak{k}((0,\infty))$ is contained in $R_2$, then $\mathfrak{k}((0,\infty))$ intersects the ball $B(y,\rho)$ by construction. On the other hand, if $\mathfrak{k}((0,\infty))$ is contained in $R_1$, then by construction $\mathfrak{k}((0,\infty))$ intersects $B(\gamma_i(y),\rho)$. This is due to the segment being below $\mathfrak{q}_i^+$ and above $\mathfrak{q}_i^-$, both of which intersect $B(\gamma_i(y),\rho)$.

Therefore, since $\mathfrak{k}$ was chosen arbitrary, all geodesics intersecting $T$, also intersect $\Gamma(B(y,\rho))$. Thus, all geodesics intersect $\Gamma(B(y,\rho))$, and hence, the open ball $B(y,\rho)$ satisfies the conditions in \Cref{Thrm-Delone if and only if one sided tangency} and so $S^+(\ell,\rho,y)$ is chaotic Delone.
\end{proof}

We conclude this section by proving the following result of the set of lengths of the $S^+(\ell,\rho,x)$. For this we introduce the following notation. For a point set $\Lambda\subset \mathbb{R}$ we define the \textit{set of lengths of $\Lambda$} as
    \begin{align*}
    \mathcal{L}_{\Lambda} = \{ x-y : x,y\in\Lambda, \, x > y \; \text{and} \; (y,x) \cap \Lambda = \emptyset
    \}.
    \end{align*}

\begin{Thm}\label{Prop:lengths_of_gaps}
If $x \in \tau^{\circ}$ and $\rho \in \mathbb{R}$ positive satisfy the conditions of \Cref{Thrm-Delone if and only if one sided tangency}, then $\mathcal{L}_{S^+(\ell,\rho,x)}$ is countably infinite.
\end{Thm}

\begin{proof}
    As $\Gamma$ is discrete, $S^+(\ell,\rho,x)$ is a countable, and hence $\mathcal{L}_{S^+(\ell,\rho,x)}$ is countable. Therefore, it is sufficient to show that $\mathcal{L}_{S^+(\ell,\rho,x)}$ is not finite.

   Assume for a contradiction that $\mathcal{L}_{S^+(\ell,\rho,x)} = \{L_1,\cdots, L_N\} \subset \mathbb{R}$, for some $N\in\mathbb{N}$.
   Let $k$ be a geodesic in $\mathbb{D}$ which lifts to a closed geodesic on $\Sigma_{\Gamma}$. By \Cref{lem:Chaotic Behaviour of geodesic} (2), we have that $S^+(k,\rho,x)$ is $(\epsilon,\delta)$-Delone.  Moreover, we claim that 
    \begin{align*}
    \mathcal{L}_{S^+(k,\rho,x)}\subseteq \mathcal{L}_{S^+(\ell,\rho,x)}.
    \end{align*}    
   To see the claim, observe that for each pair of consecutive points $x, y \in S^+(k,\rho,x)$ and each $s>0$, there is a pair of consecutive points $x', y' \in S^+(\ell,\rho,x)$ with $\lvert \lvert x - y \rvert - \lvert x' - y' \rvert \rvert \leq 1/s$. Since the set of distances $\mathcal{L}_{S^+(\ell,\rho,x)}$ is assumed to be finite and $S^+(k,\rho,x)$ is periodic it follows that $\mathcal{L}_{S^+(k,\rho,x)}\subseteq \mathcal{L}_{S^+(\ell,\rho,x)}$.

    Thus, for all geodesics $k$ which lift to a closed geodesic on $\Sigma_{\Gamma}$, we have that $\mathcal{L}_{S^+(k,\rho,x)}\subseteq \mathcal{L}_{S^+(\ell,\rho,x)}$.
    Hence, the length of a closed geodesic on the surface $\Sigma_{\Gamma}$ is contained in the set 
    \begin{align*}
    L_1\mathbb{N}\oplus\cdots\oplus L_N\mathbb{N}.
    \end{align*}   
    This however contradicts {\cite[Proposition 1.8]{realspectrumcompactificationcharacter}}, which implies that the length spectrum of a non-elementary Fuchsian group contains an infinite set of lengths that is independent over $\mathbb{Q}$.
\end{proof}

\section{Applications to cocompact triangle groups}

We now look to apply our work to cocompact Fuchsian triangle groups. For a cocompact Fuchsian triangle group $\Gamma$, regardless of our choice of fundamental domain, there exists a point $x\in\tau$ such that the closed ball $\overline{B}(x,\mu_x)$ is tangent to all sides of $\tau$. Indeed, as discussed in \Cref{sec:Chaotic_Delone}, this is a natural point to consider.
Moreover, since $\Gamma$ satisfies the setting in \Cref{sec:Chaotic_Delone}, we can apply \Cref{Nice condition for aperiodic delone sets} and find chaotic Delone sets by studying the behaviour of the extended sides of the chosen fundamental domain $\mathcal{F}$ of $\Gamma$. We note that this desired behaviour is controlled by the signature of $\Gamma$. This leads to our main result (\Cref{T:MainResult}), the proof of which requires the following lemmas that build on \cite[Theorem 1]{schmidt2024continuous}.

\begin{Lem}\label{l:two_odd_signatures_QuadDomain}
    Let $\Gamma$ be a cocompact triangle group with signature $(m_1, m_2, m_3)$ and $\mathcal{F}$ be a quadrilateral fundamental domain of $\Gamma$. If there are at least two odd numbers in the signature, then the extensions of all sides of $\mathcal{F}$ intersect $\Gamma(\mathcal{F}^\circ)$.
\end{Lem}

\begin{proof}
Let $\mathcal{F}$ be the quadrilateral fundamental domain of $\Gamma$ described in \Cref{sec:QFD}. Suppose that $m_1$ and $m_2$ are odd, and that $m_3$ is even. Since $m_1$ is odd, applying $T_4^{(m_1+1)/2}$ to the geodesic $\mathfrak{r}_4$ containing the side $r_4$, we see that the extension of $r_4$ at the vertex $v_4$ intersects $\Gamma(\mathcal{F}^\circ)$. Since $r_4$ and $r_3$ are paired, the extension of $r_3$ at $v_4$ also meets $\Gamma(\mathcal{F}^\circ)$. Since $m_2$ is also odd, a similar argument can be applied to the extensions of the sides $r_1$ and $r_2$ at the vertex $v_2$. Hence, it suffices to show that the side extensions at $v_1$ and $v_3$ intersect $\Gamma(v_4)$ or $\Gamma(v_2)$. By our arguments directly following \eqref{eq:QuadDomainPair}, we have that $T_2T_4$ fixes $v_1$ while rotating the side $r_4$ by $2\pi/{m_3}$ clockwise. Thus, $(T_2T_4)^{m_3/2}(r_4)$ will be on the extension of $r_4$ at the vertex $v_1$, and hence, the extension of $r_4$ at the vertex $v_1$ intersects $\Gamma(v_4)$. Similarly, $T_3T_1$ fixes $v_1$ and rotates the side $r_1$ by $2\pi/{m_3}$ counterclockwise. Therefore, $(T_3T_1)^{m_3/2}(r_1)$ will be on the extension of $r_1$ at the vertex $v_1$ and so, the extension of $r_1$ at the vertex $v_1$ intersects $\Gamma(v_2)$. By symmetry, the claim holds for the remaining extensions of the sides $r_2$ and $r_3$ at the vertex $v_3$.  

Now, without loss of generality, suppose that $m_1$ and $m_3$ are odd and that $m_2$ is even. From the argument above, the extensions of $r_4$ and $r_3$ at $v_4$ meets $\Gamma(\mathcal{F}^\circ)$. Thus, it suffices to show that all the other extensions meet $\Gamma(v_4)$. Applying $(T_2)^{m_2/2}$ to $\mathcal{F}$, we rotate the fundamental domain by $\pi$ counterclockwise at $v_2$, obtaining $\mathcal{F}'$, and hence, the initial part of the extension of $r_2$ at $v_2$ is the image of $r_2$. Let $v_3'=(T_2)^{m_2/2}(v_3)$ and $r_3'=(T_2)^{m_2/2}(r_3)$. Since $T_1T_3$ and $T_4T_2$ are rotations at $v_3$ by $2\pi/{m_3}$, there is also a rotation $\gamma\in\Gamma$ at $v_3'$ by $2\pi/{m_3}$ clockwise. Thus, applying $\gamma^{(m_3-1)/2}$ to $\mathcal{F}'$, we rotate $\mathcal{F}'$ by $\pi-\pi/{m_3}$ at $v_3'$, and find that the initial part of the extension of $r_2$ at $v_3'$ is the image of $r_3'$. Hence, the extension of $r_2$ at $v_2$ meets $\Gamma(v_4)$. By the side pairing relationships, the same is true for the extension of $r_1$ at $v_2$. As a result, it suffices to show that the extensions at $v_1$ and $v_3$ intersect $\Gamma(v_2)$ or $\Gamma(v_4)$. To this end, consider the extension of $r_4$ at $v_1$. Applying $(T_2T_4)^{(m_3-1)/2}$ to $\mathcal{F}$, we rotate the fundamental domain by $\pi-\pi/{m_3}$ clockwise at $v_1$, and find that the initial part of the extension of $r_4$ at $v_1$ is the image of $r_1$. Thus, the extension of $r_4$ at $v_1$ meets the $\Gamma(v_2)$. Similarly, applying $(T_3T_1)^{(m_3-1)/2}$ to $\mathcal{F}$, we rotate the fundamental domain by $\pi-\pi/{m_3}$ counterclockwise at $v_1$, and find that the initial part of the extension of $r_1$ at $v_1$ is the image of $r_4$. Therefore, the extension of $r_1$ at $v_1$ intersects $\Gamma(v_4)$. Again, by symmetry, a similar result holds for the extensions of the sides $r_2$ and $r_3$ at the vertex $v_3$.

Finally, suppose $m_1$, $m_2$ and $m_3$ are all odd.  
From the first case above, it follows that the extensions of the sides at $v_2$ and $v_4$ intersect $\Gamma(\mathcal{F}^\circ)$. For the remaining extensions, a similar argument to that in the second case applies, yielding the desired result.
\end{proof}
  
\begin{Lem}\label{l:hexDomain}
    Let $\Gamma$ be a cocompact triangle group with signature $(m_1, m_2, m_3)$ and $\mathcal{F}$ be a hexagonal fundamental domain of $\Gamma$. The extensions of all the sides of $\mathcal{F}$ intersect $\Gamma(\mathcal{F}^\circ)$.
\end{Lem}

\begin{proof}
Independent of the signature of $\Gamma$, by the side pairing relationships, we have that \mbox{$U_1(a_1)=a_3$,} $U_2(a_2)=a_1$, $U_3(a_2)=a_3$, resulting in $a_3\in\mathcal{F}\cap U_1(\mathcal{F})\cap U_3
(\mathcal{F})$. Since we also have $\sum_i\theta_i=2\pi$ with $\theta_i<\pi$ for each $i$, where $\theta_i$ denotes the internal angle at the vertices $a_i$, the extensions of the sides $t_3$ and $s_1$ at $a_3$ intersect the $\Gamma(\mathcal{F}^\circ)$. A similar argument can be applied for the extensions of the sides at the vertices $a_1$ and $a_2$. 

Consider the vertex $v_i$ where $m_i$ is even. Applying $U_i^{m_i/2}$ to $\mathcal{F}$, we obtain a rotation of $\mathcal{F}$ at $v_i$ by $\pi$ and find that the extensions of $t_i$ and $s_{i}$ at $v_{i}$ has an image of $a_{i}$ and $a_{i-1}$ respectively. Since the extensions at the vertices $a_i$ intersect $\Gamma(\mathcal{F}^\circ)$, it follows that the extensions of the sides at $v_i$ intersect $\Gamma(\mathcal{F}^\circ)$.

Now, suppose $m_i$ is odd. In this case, $U_i^{(m_i+1)/2}$ rotates $t_i$ and $s_i$ by $\pi+\pi/{m_i}$ at $v_i$. Thus, the extensions of $t_i$ and $s_i$ at $v_i$ intersects $\Gamma(\mathcal{F}^\circ)$. 
\end{proof}

\begin{proof}[Proof of \Cref{T:MainResult} (1)]
    We observe that for a triangle group $\Gamma$ with at least two odd numbers in the signature, \Cref{l:two_odd_signatures_QuadDomain} verifies that the extended sides of the polygon $\mathcal{F}$ intersect $\Gamma(\mathcal{F}^\circ)$. Hence applying \Cref{Nice condition for aperiodic delone sets}, we can see that $S^+(\ell,\rho,x)$ is chaotic Delone.

    For the forward direction we consider the case when there is at most one odd number in the signature. Letting $\mathcal{F}$ be the quadrilateral fundamental domain constructed in \Cref{cocompactTriGp}. It is shown \mbox{\cite[Theorem 1]{schmidt2024continuous}} that when all $m_i$ are even or when only $m_3$ is odd, then the boundary of $\mathcal{F}$ lifts to closed geodesics on $\Sigma_{\Gamma}$, which can not intersect $B(\pi(x),\rho)$ and hence contradicting \Cref{Thrm-Delone if and only if one sided tangency}.
    
    The remaining case is when either $m_1$ or $m_2$ is odd. Without loss of generality, suppose that $m_1=2k+1$ and that $m_2$ and $m_3$ are even. Consider the side $r_1$ of $\mathcal{F}$ emanating from the vertex $v_2$. We follow the unit tangent vectors along this side until $v_1$ and then we apply $(T_2T_4)^{m_3/2}$, which rotates it by $(m_3/2)(2\pi/m_3)=\pi$. Following, this path back along $r_{1}$ until $v_2$, we obtain a unit tangent vector which has the opposite direction from the departing unit tangent vector. We now apply $(T_2)^{m_2/2}$ so as to turn the unit tangent vector to its initial direction. Lifting this path to $\Sigma_{\Gamma}$ yields a closed geodesic contained in the lift of the boundary of $\mathcal{F}$, which can not intersect $B(\pi(x),\rho)$. Hence contradicting \Cref{Thrm-Delone if and only if one sided tangency}. 
\end{proof}

\begin{proof}[Proof of \Cref{T:MainResult}\,(2) and (3)]
    \Cref{T:MainResult}\,(2) is a direct application of \Cref{l:hexDomain} in tandem with \Cref{Nice condition for aperiodic delone sets}, and \Cref{T:MainResult}\,(3) is an application of \Cref{Prop:lengths_of_gaps}.
\end{proof}

\section*{Acknowledgements}

The authors thank Gregory J. Chaplain (Department of Physics and Astronomy, University of Exeter) and Andrew Mitchell (Department of Mathematical Sciences, Loughborough University) for valuable discussions. Richard A. Howat thanks EPSRC Doctoral Training Partnership, the University of Birmingham and the University of Exeter for financial support during the development and writing of this article. Ay\c{s}e Y{\i}ltekin-Karata\c{s} is extremely grateful to LMS Scheme 5 grants 52511 and 52406 and the University of Exeter for supporting several research visits to collaborate on this work.

\bibliographystyle{abbrv}
\bibliography{Ref}

\begin{thebibliography}{10}

\bibitem{ChaoticDeloneSet}
J.~A. \'Alvarez~L\'opez, R.~Barral~Lij\'o, J.~Hunton, H.~Nozawa, and J.~R.
  Parker.
\newblock Chaotic {D}elone sets.
\newblock {\em Discrete Contin. Dyn. Syst.}, 41(8):3781--3796, 2021.

\bibitem{Aperiodic_Order_Vol1}
M.~Baake and U.~Grimm.
\newblock {\em Aperiodic order. {V}ol. 1}, volume 149 of {\em Encyclopedia of
  Mathematics and its Applications}.
\newblock Cambridge University Press, Cambridge, 2013.

\bibitem{BAAKE_LENZ_2004}
M.~Baake and D.~Lenz.
\newblock Dynamical systems on translation bounded measures: pure point
  dynamical and diffraction spectra.
\newblock {\em Ergodic Theory Dynam. Systems}, 24(6):1867–1893, 2004.

\bibitem{beardon2012geometry}
A.~F. Beardon.
\newblock {\em The geometry of discrete groups}, volume~91.
\newblock Springer, 2012.

\bibitem{LocalRubberTop}
J.~Bellissard, D.~J.~L. Herrmann, and M.~Zarrouati.
\newblock Hulls of aperiodic solids and gap labeling theorems.
\newblock In {\em Directions in mathematical quasicrystals}, volume~13 of {\em
  CRM Monogr. Ser.}, pages 207--258. American Mathematical Society, Providence,
  RI, 2000.

\bibitem{realspectrumcompactificationcharacter}
M.~Burger, A.~Iozzi, A.~Parreau, and M.~B. Pozzetti.
\newblock The real spectrum compactification of character varieties:
  characterizations and applications, 2021.

\bibitem{PhysRevLett.131.177001}
B.~Davies, G.~J. Chaplain, T.~A. Starkey, and R.~V. Craster.
\newblock Graded quasiperiodic metamaterials perform fractal rainbow trapping.
\newblock {\em Phys. Rev. Lett.}, 131:177001, 2023.

\bibitem{EinsiedlerWard2011}
M.~Einsiedler and T.~Ward.
\newblock {\em Ergodic theory with a view towards number theory}, volume 259 of
  {\em Graduate Texts in Mathematics}.
\newblock Springer, 2011.

\bibitem{GL_1989}
C.~Godrèche and J.~M. Luck.
\newblock Quasiperiodicity and randomness in tilings of the plane.
\newblock {\em J. Stat. Phys.}, 55:1--28, 1989.

\bibitem{KatokHasselblatt1995}
A.~Katok and B.~Hasselblatt.
\newblock {\em Introduction to the modern theory of dynamical systems},
  volume~54 of {\em Encyclopedia of Mathematics and Its Applications}.
\newblock Cambridge University Press, Cambridge, 1995.

\bibitem{Katok1992FuchsianGroups}
S.~Katok.
\newblock {\em Fuchsian groups}.
\newblock Chicago Lectures in Mathematics. University of Chicago Press, 1992.

\bibitem{ErgodicThryDGroups}
P.~J. Nicholls.
\newblock {\em The ergodic theory of discrete groups}, volume 143 of {\em
  London Mathematical Society Lecture Note Series}.
\newblock Cambridge University Press, Cambridge, 1989.

\bibitem{schmidt2024continuous}
T.~A. Schmidt and A.~Y{\i}ltekin-Karata{\c{s}}.
\newblock Continuous deformation of the {B}owen-{S}eries map associated to a
  cocompact triangle group.
\newblock {\em Geom. Dedicata}, 218(3):60, 2024.

\end{thebibliography}

\appendix
\section{}

Throughout the appendix we use the notation and setting of \Cref{sec:Chaotic_Delone}. 
We remark, that the statements of \cite[Lemmas 4.4, 4.8, 4.9 and 4.10]{ChaoticDeloneSet} are given for point sets on the geodesic in $\mathbb{D}$, whereas the statements of our results transforms the point set via an isometry to $\mathbb{R}$. For ease of exposition, we also abuse notation and use the same symbol for a geodesic and its parametrisation.

\begin{proof}[{Proof of \Cref{lem:Chaotic Behaviour of geodesic}}\,(1)]
    Let $r \in \mathbb{R}$ positive and let $\overline{\mathcal{N}}^+(\ell([-r,r]),\rho)$ be as defined in \eqref{eq:N+} with the geodesic replaced by the geodesic segment $\ell([-r,r])$. By \Cref{Lem:trans anosov flow} there exists a geodesic $k$ in $\mathbb{D}$ which lifts to a closed geodesic on $\Sigma_{\Gamma}$ such that
    \begin{enumerate}[label=(\roman*)]
        \item\label{cond(i)} the hyperbolic segment $k([-r,r])$ is contained in the positive component of $\mathcal{N}(\ell,\rho)\setminus \ell$ with respect to the normal bundle of $\ell$, and $\ell([-r,r])$ is contained in the positive component of \mbox{$\mathcal{N}(k,\rho)\setminus k$} with respect to the normal bundle of $k$;
    \end{enumerate}
    which, in tandem with the fact that $Z=\overline{\mathcal{N}}^+(\ell([-r,r]),\rho) \cap \Gamma(x)$ is finite, allows us to choose $k$ where
    \begin{enumerate}[label=(\roman*)]
        \setcounter{enumi}{1}
        \item\label{cond(ii)} the set $Z$ is contained in $\overline{\mathcal{N}}^+(k([-r,r]),\rho)\cap \overline{\mathcal{N}}^+(\ell([-r,r]),\rho)$;
        \item\label{cond(iii)} for all $y\in \ell([-r,r])$, we have \mbox{$d(i(y),y)<1/(2r)$}, where $i$ denotes the unique orientation reversing isometry mapping $\ell(0)$ to $k(0)$ preserving the first assumption;
        \item\label{cond(iv)} for all $z\in Z$, we have $d(p_\ell(z),p_k(z))<1/(2r)$.
    \end{enumerate}
    For $z\in Z$, by the triangle inequality and the assumptions~\ref{cond(iii)} and~\ref{cond(iv)}
    \begin{align*}
    d(i(p_\ell(z)),p_k(z)) \leq d(i(p_\ell(z)),p_\ell(z))+d(p_\ell(z),p_k(z))<1/r.
    \end{align*}
    Since $i$ is an isometry with $\ell(0)=i(k(0))$, we have $i(\ell(S^+(k,\rho,x)))=k(S^+(k,\rho,x))$.
    Hence, for $u$ in $S^+(\ell,\rho,x)\cap [-r,r]$, letting $z\in Z$ be the point with $p_\ell(z)=\ell(u)$, we have $d(i(p_\ell(z)),p_k(z))<1/r$.
   By assumption~\ref{cond(ii)}, there exists a point  $y\in S^+(k,\rho,x)\cap [-r, r]$ with $k(y)=p_k(z)$. 
   Therefore, it follows that $d(\ell(u),k(y))=d(p_\ell(z),p_k(z))<1/r$, which implies $u\in S^+(k,\rho,x)+[-1/r,1/r]$, or in other words, $S^+(\ell,\rho,x)\cap [-r,r]\subseteq S^+(k,\rho,x)+[-1/r,1/r]$. An analogous argument shows that $S^+(k,\rho,x)\cap [-r,r]$ is contained in $S^+(\ell,\rho,x)+[-1/r,1/r]$, which implies $(S^+(\ell,\rho,x),S^+(k,\rho,x))\in N_r$.
    \end{proof}

    \begin{proof}[{Proof of \Cref{lem:Chaotic Behaviour of geodesic}}\,(2)] 
    Let $r \in \mathbb{R}$ be positive and let $k$ be a given geodesic in $\mathbb{D}$. 
    Recall that the orbit of $(x_{0}, v_{0})$ 
    under the geodesic flow is dense in $\operatorname{ST}(\Sigma_{\Gamma})$, where $x_0$ lies on the lift of $\ell$.
    By this transitivity of $(x_0, v_0)$ and since $Z'=\overline{\mathcal{N}}^+(k([-r,r]),\rho)\cap \Gamma(x)$ is finite, there exists 
    $\gamma\in\Gamma$  
    with $t=\gamma(\ell)$ such that
    \begin{enumerate}[label=(\roman*)]
        \item the hyperbolic segment $k([-r,r])$ is contained in the positive component of $\mathcal{N}(t,\rho)\setminus t$ with respect to the normal bundle of $t$ and $t([-r,r])$ is contained in the positive component of $\mathcal{N}(k,\rho)\setminus k$ with respect to the normal bundle of $k$;
        \end{enumerate}
which allows us to choose $\gamma$ such that
        \begin{enumerate}[label=(\roman*)]
        \setcounter{enumi}{1}
       \item  the set $Z'$ is contained in $\overline{\mathcal{N}}^+(k([-r,r]),\rho)\cap \overline{\mathcal{N}}^+(t([-r,r]),\rho)$;
        \item for all $y\in t([-r,r])$, we have $d(j(y),y)<1/(2r)$, where $j$ denotes the unique orientation reversing isometry mapping $t(0)$ to $k(0)$ preserving the first assumption;
        \item for all $z\in Z'$ we have that $d(p_k(z),p_t(z))<1/(2r)$.
     \end{enumerate}
    Therefore, using arguments identical to those in the proof of \Cref{lem:Chaotic Behaviour of geodesic}\,(1), one may conclude that $(S^+(t,\rho,x),S^+(k,\rho,x))\in N_r$.  Hence, letting  
    $a\in\mathbb{R}$ be such that $\gamma(\ell(a))=t(0)$, it follows that $S^+(t,\rho,x)=S^+(\ell,\rho,x)-a$.
\end{proof}

\begin{proof}[{Proof of \Cref{lem: ap}\,(1)}]
    Let $y,y'\in S^+(\ell,\rho,x)$ be given such that $y\neq y'$.
    Letting $z,z'\in \overline{\mathcal{N}}^+(\ell,\rho)$ denote the points such that $p_\ell(z)=\ell(y)$ and $p_\ell(z')=\ell(y')$, we observe the following inequality
        \begin{align*}
        2\operatorname{inj}(\Gamma,x)\leq d(z,z')\leq d(z,\ell(y))+d(\ell(y),\ell(y'))+d(\ell(y'),z')\leq 2\rho + d(\ell(y),\ell(y')),
        \end{align*}
    which implies that $|y-y'|= d(\ell(y),\ell(y'))\geq 2(\operatorname{inj}(\Gamma,x)-\rho)$.
\end{proof}

\begin{proof}[{Proof of \Cref{lem: ap} (2)}]
    We first prove the result when $\rho>\operatorname{inj}(\Sigma,x)$, and then describe how the proof can be adapted to obtain the result when $\rho=\operatorname{inj}(\Sigma,x)$.
    Let $z,w\in \Gamma(x)$ be 
    such that $d(z,w)=2\operatorname{inj}(\Gamma,x)$.
    Let $\delta\in\mathbb{R}$ be positive and fixed, and let $h=h_{z,w}$ denote the geodesic segment connecting $z$ and $w$. Denote the midpoint of $h$ taken with respect to $d$ by $h_m$. We define $v$ to be the unit normal vector of $h$ at $h_m$. Let $k$ be the geodesic with $k(0)=h_m$ and $k'(0)=v$. By an analogous argument given in the proof of \Cref{lem:Chaotic Behaviour of geodesic}\,(2), given $\epsilon>0$, there exists $\gamma\in \Gamma$ and $a\in\mathbb{R}$ such that $\lvert\gamma'(\ell(a))-v\rvert<\epsilon$ and $\max\{d(\gamma(\ell),z),d(\gamma(\ell),w)\}<\operatorname{inj}(\Gamma,x)+\epsilon$.
    This in tandem with the fact that  $p_k(z)=p_k(w)=k(0)=h_m$, 
    we may choose $\gamma\in\Gamma$ and $a\in\mathbb{R}$ with $\gamma(\ell) \neq k$ and $\max\{d(p_{\gamma(\ell)}(z),p_k(z)),d(p_{\gamma(\ell)}(w),p_k(w))\}<\delta/2$.  Thus, it follows that $0 < d(p_{\gamma(\ell)}(w),p_{\gamma(\ell)}(z))<\delta$. 
    Moreover, since $\gamma$ is an isometry we have that $d(p_{\gamma(\ell)}(w),p_{\gamma(\ell)}(z)) = d(p_{\ell}(\gamma^{-1}(w))$, $p_{\ell}(\gamma^{-1}(z)))$, $d(\gamma(\ell),z) = d(\ell,\gamma^{-1}(z))$ and $d(\gamma(\ell),w) = d(\ell,\gamma^{-1}(w))$, yielding the required result.

    To obtain the result when $\rho=\operatorname{inj}(\Sigma,x)$ we observe that in this setting $d(\gamma(\ell),z)$ or $d(\gamma(\ell),w)$ may be greater that $\rho$. To rectify this issue, we instead of using the vector $v$, we take a small rotation $v^*$ of $v$, and choose a geodesic $t$ such that $t(0)=h_m$ and $t'(0)=v^*$. Since the geodesic flow along the lift of $\ell$ is dense in $\operatorname{ST}(\Sigma_\Gamma)$, there exists $\gamma\in \Gamma$ and $a\in\mathbb{R}$ with $\lvert\gamma'(\ell(a))-v^*\rvert$ and $d(\gamma(\ell(a)),h_m)$ arbitrarily small, and $\gamma(\ell)\neq k$.
    The above argument, for the case when $\rho > \operatorname{inj}(\Sigma,x)$ is now repeated, replacing $v$ with $v^*$ and $k$ with $t$, noting that we can choose $v^*$ such that $d(p_t(z),p_t(w))$ is positive and arbitrarily small, since $v^*$ is not normal to $h$ at $h_{m}$, and $d(\gamma(\ell),z),d(\gamma(\ell),w)<\rho$.
\end{proof}

\begin{proof}[{Proof of \Cref{lem: appp}}]
    For $(y, v) \in \operatorname{ST}(\Sigma_\Gamma)$, 
    let $k_{(y, v)}$ denote the unique geodesic with $k_{(y,v)}(0) = y$ and $k'_{(y,v)}(0) = v$. 
    We define the function $f:\operatorname{ST}(\Sigma_\Gamma)\rightarrow \mathbb{R}$ by $f((y,v))= \inf\{|t|:k_{(y,v)}(t)\in B(\pi(x),\rho)\}$.
    Since all geodesics in $\Sigma_\Gamma$ intersect $B(\pi(x),\rho)$ by assumption, 
    this function is well defined. Moreover, it
    is upper semi-continuous, and since $\operatorname{ST}(\Sigma_\Gamma)$ is compact, the function
    $f$ is bounded. It follows that the set $S^+(\ell,\rho,x)$ is $\varepsilon$-relatively dense for some $\varepsilon>0$.
\end{proof}

\begin{proof}[{Proof of \Cref{lem: apppp}}]
    Assume that $S^+(\ell,\rho,x)$ is periodic with minimal period $\omega\in\mathbb{R}$, and let $k$ be a closed geodesic. \Cref{lem:Chaotic Behaviour of geodesic} (2) implies that $S^+(\ell,\rho,x)$ has the same minimal period as $S^+(k,\rho,x)$. Since $k$ has finite length $l(k)$ we have that $l(k)=m\omega$ for some $m\in\mathbb{N}$. Therefore, the length spectrum $L(\Gamma)=\omega \cdot \mathbb{N}$. 
    However, this is a contradiction to \cite[Proposition\,1.8]{realspectrumcompactificationcharacter}
    namely that the length spectrum of a non-elementary Fuchsian group contains an infinite set of lengths that is independent over $\mathbb{Q}$.
\end{proof}

\end{document}